\documentclass[11pt]{amsart}

\usepackage{amsmath,amsthm,amsfonts,amssymb,bm,graphicx}

\usepackage
{hyperref}
\hypersetup{colorlinks=true,citecolor=blue,linkcolor=blue,urlcolor=blue}

\theoremstyle{plain}
\newtheorem{theorem}{Theorem}
\newtheorem*{theorem*}{Theorem}
\newtheorem{conj}[theorem]{Conjecture}
\newtheorem*{conj*}{Conjecture}
\newtheorem{lemma}[theorem]{Lemma}

\newtheorem{prop}[theorem]{Proposition}

\theoremstyle{remark}

\newcommand{\Tr}{{\mathrm Tr}}

\numberwithin{equation}{section}

\def\N{\mathbb N}
\def\Z{\mathbb Z}
\def\Q{\mathbb Q}
\def\O{\mathcal O}

\begin{document}

\author{Magdal\'ena Tinkov\'a}
\author{Paul Voutier}

\title{Indecomposable integers in real quadratic fields}

\address{Charles University, Faculty of Mathematics and Physics, Department of Algebra,
Sokolovsk\'{a}~83, 18600 Praha 8, Czech Republic}
\email{tinkova.magdalena@gmail.com}

\address{London, UK}
\email{paul.voutier@gmail.com}


\thanks{The first author was supported by Czech Science Foundation (GA\v{C}R), grant 17-04703Y, by the Charles University, project GA UK No. 1298218, by Charles University Research Centre program UNCE/SCI/022, and by the project SVV-2017-260456.}

\begin{abstract}
In 2016, Jang and Kim stated a conjecture about the norms of indecomposable
integers in real quadratic number fields $\Q \left( \sqrt{D} \right)$ where $D>1$
is a squarefree integer. Their conjecture was later disproved by Kala
for $D \equiv 2 \bmod 4$. We investigate such indecomposable integers in greater
detail. In particular, we find the minimal $D$ in each congruence class $D \equiv 1,2,3 \bmod 4$
that provides a counterexample to the Jang-Kim Conjecture; provide infinite
families of such counterexamples; and state a refined version of the Jang-Kim
conjecture. Lastly, we prove a slightly weaker version of our refined conjecture that is of the correct order of magnitude, showing the Jang-Kim Conjecture is only wrong by at most $O \left( \sqrt{D} \right)$.
\end{abstract}

\setcounter{tocdepth}{2}
\maketitle 

\section{Introduction}

The ring of algebraic integers $\O_{K}$ of number field $K$ is one of the key
objects studied in algebraic number theory, and its additive structure sometimes
plays a surprisingly important role. For example, in totally real number fields,
we can focus on the semiring of totally positive elements of $\O_{K}$, denoted
by $\O_{K}^{+}$, and define the subset of so-called indecomposable integers,
i.e., elements of $\O_{K}^{+}$ which cannot be expressed as a sum of two elements
of $\O_{K}^{+}$. Several interesting applications of indecomposable integers to
universal quadratic forms (i.e., positive quadratic forms over $\O_{K}$ that
represent all elements of $\O_{K}^{+}$) have been recently developed by Blomer,
Kala or Kim \cite{BK,Ka2,BK2,Kim}, although they have been also used by Siegel
\cite{Si} already in 1945 in a similar context.

All indecomposable integers in real quadratic fields $\Q \left( \sqrt{D} \right)$,
where $D>1$ is a squarefree integer, can be nicely described using the continued
fraction expansion of $\sqrt{D}$ or $\left( \sqrt{D}-1 \right)/2$ in the cases $D \equiv 2,3 \bmod 4$
or $1 \bmod 4$, resp. -- see Lemma~\ref{lem:ds} below.

Currently, we do not have such a characterization for the fields of higher degrees. Nevertheless, we can mention a partial result given by \v{C}ech, Lachman, Svoboda, Zemkov\'a and the first author \cite{CLSTZ} considering biquadratic fields, which have degree $4$.

Using the above description for real quadratic fields, Dress and Scharlau
\cite[Theorem~3]{DS} deduced an upper bound on the norm of quadratic indecomposable
integers. This result was later refined by Jang and Kim by proved the following
result.

\begin{theorem}[\cite{JK} Theorem~5]
\label{thm:jk}
Let $D>1$ be a squarefree integer and let $-N$ be the largest negative
norm of the algebraic integers in $K=\Q \left( \sqrt{D} \right)$. Then
\[\def\arraystretch{2.2}
N(\alpha) \leq
\left\{
\begin{array}{ll}
	\displaystyle\frac{D-1}{4N} & \text{if $D \equiv 1 \bmod 4$}, \\
	\displaystyle\frac{D}{N}  & \text{if $D \equiv 2,3 \bmod 4$},
\end{array}
\right.
\]
for all indecomposable $\alpha\in \O_{K}$. 
\end{theorem}

It is important to mention that equality always holds in this result if the norm
of the fundamental unit in $K$ is equal to $-1$. Moreover, except for the mentioned
$N=1$ cases, this bound is lower than the bound given in \cite{DS}.

In the same paper, Jang and Kim also stated a conjecture improving their upper bound.

\begin{conj}[\cite{JK}]
\label{conj:JK}
Let $D$ and $N$ be as in Theorem~$\ref{thm:jk}$ and $a$ be the smallest nonnegative
rational integer such that $N$ divides $D-a^{2}$. Then
\[\def\arraystretch{2.2}
N(\alpha) \leq
\left\{
\begin{array}{ll}
	\displaystyle\frac{D-a^{2}}{4N} & \text{if $D \equiv 1 \bmod 4$}, \\
	\displaystyle\frac{D-a^{2}}{N}  & \text{if $D \equiv 2,3 \bmod 4$},
\end{array}
\right.
\]
for all indecomposable $\alpha\in \O_{K}$. 
\end{conj}

This conjecture was disproved by Kala \cite[Theorem~4]{Ka} in the case $D \equiv 2 \bmod 4$
by giving a counterexample.

The main aim of this paper is to study the norms of indecomposable integers in
greater detail. In particular, we find the minimal $D$ in each congruence class
$D \equiv 1,2,3 \bmod 4$ that provides a counterexample to the Jang-Kim Conjecture;
provide infinite families of such counterexamples; and state a refined version
of the Jang-Kim Conjecture (see Conjecture~\ref{conj:MP} below). If true, our families of examples show that our
refined conjecture is best-possible.

\begin{conj}
\label{conj:MP}
Let $D$ and $N$ be as in Theorem~$\ref{thm:jk}$. Let $a$ be the smallest
nonnegative rational integer such that $a^{2} \equiv D \bmod N$ if $D \equiv 2,3 \bmod 4$
and such that $a^{2} \equiv D \bmod 4N$ if $D \equiv 1 \bmod 4$. Then
\[
\def\arraystretch{2.2}
N(\alpha) <
\left\{
\begin{array}{ll}
	\displaystyle\frac{D-a^{2}}{4N}+\frac{\sqrt{D}}{8} & \text{if $D \equiv 1 \bmod 4$}, \\
	\displaystyle\frac{D-a^{2}}{N}+\frac{\sqrt{D}}{4} & \text{if $D \equiv 2,3 \bmod 4$},
\end{array}
\right.
\]
for all indecomposable $\alpha \in \O_{K}$. 
\end{conj}

Notice that in addition to the extra term we have added, we have also corrected the main term for $D \equiv 1 \bmod 4$, as the definition of $a$ in Jang-Kim's formulation of the conjecture in that case does not seem to be right.

We have also been able to prove a result in the direction of our refined conjecture; namely that the Jang-Kim Conjecture is only wrong by at most $O \left( \sqrt{D} \right)$. We shall prove the following result.

\begin{theorem}
\label{thm:MP}
Let $a$, $D$ and $N$ be as in Conjecture~$\ref{conj:MP}$. Then
\[
\def\arraystretch{2.2}
N(\alpha) <
\left\{
\begin{array}{ll}
	\displaystyle\frac{D-a^{2}}{4N}+\sqrt{D} & \text{if $D \equiv 1 \bmod 4$}, \\
	\displaystyle\frac{D-a^{2}}{N}+2\sqrt{D}  & \text{if $D \equiv 2,3 \bmod 4$},
\end{array}
\right.
\]
for all indecomposable $\alpha \in \O_{K}$.
\end{theorem}

Section~\ref{sec:pre} is devoted to study of indecomposable integers in $\Q \left( \sqrt{D} \right)$
where $D>1$ is a squarefree integer. We will introduce some notation as well as basic facts about continued fraction expansions and about algebraic integers in the considered real quadratic fields.

In Sections~\ref{sec:ni} and \ref{sec:mi}, we provide the estimates for the quantities used initially to find counterexamples to the Jang-Kim Conjecture and required to prove Theorem~\ref{thm:MP}.

Section~\ref{sec:calcs} contains the calculations which resulted in finding
counterexamples for $D \equiv 1,2,3 \bmod 4$. Moreover, our search is exhaustive
so the counterexamples provided here have minimal discriminants. Sections~\ref{sec:min} and \ref{sec:family} contain details about the minimal counterexamples and the families of counterexamples found, while in Section~\ref{sec:conj} we present a new conjecture that appears to be best-possible as well as the proof of Theorem~\ref{thm:MP}.

\section{Preliminaries}
\label{sec:pre}

Let $K$ be a real quadratic number field and let $\alpha'$ denote the conjugate
of $\alpha \in K$. Let $R$ be an order of discriminant $\Delta>0$ in $K$.
There exists a non-square positive integer $D$ such that $K = \Q \left( \sqrt{D} \right)$
and either $R = \Z \left[ \sqrt{D} \right]$ or $R = \Z \left[ \left( 1+\sqrt{D} \right)/2 \right]$.
For the most part we will work with $R=\O_{K}$, but since the results of
\cite{DS} are stated for any order, $R$, we do so initially here too.

Any $\alpha \in R$ is said to be \emph{totally positive}, denoted by $\alpha \succ 0$,
if both $\alpha$ and $\alpha'$ are positive. We write $R^{+}$ for the set of such
elements. We call $\alpha$ \emph{indecomposable} if it cannot be expressed as a
sum $\alpha=\beta+\gamma$ where $\beta,\gamma \in R^{+}$. This also indicates
that every element of $R^{+}$ can be written as a sum of finitely many
indecomposable integers. We will use the symbol $N(\alpha)$ to denote the norm
of any $\alpha \in K$.

We put
\[
\delta =
\left\{
\begin{array}{ll}
            \sqrt{D} & \text{if $R=\Z \left[ \sqrt{D} \right]$},\\
\displaystyle\frac{\sqrt{D}+1}{2} & \text{if $R=\Z \left[ \left( 1+\sqrt{D} \right)/2 \right]$},
\end{array}
\right.
\]
and
\[
\xi =
\left\{
\begin{array}{ll}
            \sqrt{D} & \text{if $R=\Z \left[ \sqrt{D} \right]$},\\
\displaystyle\frac{\sqrt{D}-1}{2} & \text{if $R=\Z \left[ \left( 1+\sqrt{D} \right)/2 \right]$}.
\end{array}
\right.
\]

For any continued fraction $\left[ u_{0},u_{1},\ldots \right]$ (not necessarily
a simple continued fraction), let us define
inductively
\begin{eqnarray*}
p \left( u_{0} \right) & = & u_{0}, \quad q \left( u_{0} \right)=1, \quad
p \left( u_{0}, u_{1} \right) = u_{1}u_{0}+1, \quad q \left( u_{0}, u_{1} \right) = u_{1}, \nonumber \\
p \left( u_{0}, u_{1},\ldots,u_{j} \right) & = & u_{j} p \left( u_{0},u_{1},\ldots,u_{j-1} \right) + p \left( u_{0},u_{1},\ldots,u_{j-2} \right), \\
q \left( u_{0}, u_{1},\ldots,u_{j} \right) & = & u_{j} q \left( u_{0},u_{1},\ldots,u_{j-1} \right) + q \left( u_{0},u_{1},\ldots,u_{j-2} \right), \nonumber
\end{eqnarray*}
for $j=2,3,\ldots$

Thus
\[
\left[ u_{0},u_{1},\ldots,u_{j} \right]
= \frac{p \left( u_{0},u_{1},\ldots,u_{j} \right)}{q \left( u_{0},u_{1},\ldots,u_{j} \right)}.
\]

Let $\left[ u_{0},\overline{u_{1},u_{2},\ldots,u_{s-1},u_{s}} \right]$ be the
simple continued fraction expansion of $\xi$, where $u_{s}=2u_{0}+1$ if
$d \equiv 1 \bmod 4$ and $u_{s}=2u_{0}$ if $d \equiv 2,3 \bmod 4$. This expansion
is periodic with $s$ as the length of the minimal period. Note that the sequence
$u_{1},u_{2},\ldots,u_{s-1}$ is symmetric (that is, $u_{1}=u_{s-1}$, $u_{2}=u_{s-2}$,\ldots).
We shall simplify the above notation here as follows,
\[
\frac{p_{i}}{q_{i}}=[u_{0},u_{1},\ldots,u_{i}],
\]
for $i \geq 0$. We shall also follow the usual convention of setting $p_{-1}=1$
and $q_{-1}=0$. This fraction is called the \emph{$i$-th convergent} of $\xi$
and we set
\[
\delta_{i}=p_{i}+q_{i}\delta,
\]
for $i \geq -1$.

These $\delta_{i}$ are elements of the ring $\O_{K}$.
The $\delta_{i}$'s are also related to another group of special rational integers.
We put
\[
\frac{p_{i,r}}{q_{i,r}}=\left[ u_{0},u_{1},\ldots,u_{i+1}, r \right],
\]
where $i \geq -1$ and $0 \leq r \leq u_{i+2}$ and call such rational numbers
\emph{semi-convergents}. So here
\[
\delta_{i,r}=p_{i,r}+q_{i,r}\delta.
\]

These semi-convergents
are of interest to us here due to the following result. It follows readily from
\cite[\S~16]{Pe}, but appears to have been first explicitly stated and proven
by Dress and Scharlau \cite{DS}.

\begin{lemma}
\label{lem:ds}
If $\xi=\left[ u_{0}, u_{1}, \ldots \right]$ is the simple continued fraction
expansion of $\xi$, then the indecomposable elements, $\alpha$, in $R^{+}$ are of
the form $\alpha=1$ or $\alpha=p \left( u_{0},\ldots,u_{i},r \right) + q \left( u_{0},\ldots,u_{i},r \right) \delta$
or $\alpha=p \left( u_{0},\ldots,u_{i},r \right) + q \left( u_{0},\ldots,u_{i},r \right) \delta'$
with $i \geq 0$, $i \equiv 0 \bmod 2$ and $r=1,2,\ldots, u_{i+1}$.
\end{lemma}

\begin{proof}
This is the second part of Theorem~2 in \cite{DS}.
\end{proof}

Since $\delta_{i} \succ 0$ only for $i$ odd, we will use the symbol $N_{i}$ to
denote the absolute value of the norm of $\delta_{i}$, i.e.,
$N_{i}= \left| N \left( \delta_{i} \right) \right|
=(-1)^{i+1}N \left( \delta_{i} \right)$. In what follows, let
\begin{align*}
M_{i} & = \max \left\{ N \left( \delta_{i,r} \right) ; N \left( \delta_{i} \right)>0 \text{ and } 0 \leq r \leq u_{i+2} \right\}, \\
M & = \max \left\{ M_{i} ; i \right\}, \\
N & = \min \left\{ |N(\alpha)|;\;\alpha\in\O_{K} \text{ such that } N(\alpha)<0 \right\}
= \min \left\{ N_{i};\;i \text{ even} \right\}.
\end{align*}

Sometimes we will use $M_{D}$ and $N_{D}$, or similar indexing to indicate $M$ and $N$, especially in Section~\ref{sec:family}.

Throughout this paper, we will denote by $c_{i}$ the expression
\[
c_{i}=\left[ u_{i},u_{i+1},u_{i+2},\ldots \right].
\]

Thus
\begin{equation}
\label{eq:ci}
c_{i}=\left[ u_{i},u_{i+1},u_{i+2},\ldots \right]
= u_{i}+\frac{1}{\left[ u_{i+1},u_{i+2},\ldots \right]}=u_{i}+\frac{1}{c_{i+1}}.
\end{equation}

Therefore, we also have the equation
\[
\xi=\frac{c_{i+1}p_{i}+p_{i-1}}{c_{i+1}q_{i}+q_{i-1}}.
\]

Let $T_{i}^{(a)}$ and $T_{i}^{(b)}$ be the two quantities given by
\begin{align*}
T_{i}^{(a)} &= p_{i}p_{i-1}+p_{i}q_{i-1}\Tr(\delta)+q_{i}q_{i-1}N(\delta), \\
T_{i}^{(b)} &= p_{i}p_{i-1}+p_{i-1}q_{i}\Tr(\delta)+q_{i}q_{i-1}N(\delta).
\end{align*}

These two expressions satisfy $\delta_{i+1}'\delta_{i}=T_{i+1}^{(a)}+(-1)^{i}\xi
=T_{i+1}^{(b)}+(-1)^{i}\delta=T_{i+1}^{(b)}-(-1)^{i}\xi'$. 

To estimate the values of the $M_{i}$'s and $N_{i}$'s, we will need to establish several relations involving numbers $T_{i}^{(a)}$ and $T_{i}^{(b)}$. 
In the following lemma, we will express $T_{i}^{(a)}$
and $T_{i}^{(b)}$ using $\sqrt{D}$, $c_{i+1}$ and the norm of $\delta_{i}$.
This is the generalization of Lemma~5b) of \cite{Ka} for all squarefree integers $D>1$.

\begin{lemma}
\label{lem:tiab}
Suppose $D>1$ is a squarefree integer. For each $i \in \N_{0}$ we have
\begin{align*}
T_{i}^{(a)} &= (-1)^{i+1}\delta-c_{i+1}N(\delta_{i}),\\
T_{i}^{(b)} &= (-1)^{i+1}\xi-c_{i+1}N(\delta_{i}).
\end{align*}
\end{lemma}

\begin{proof}
The main idea of the proof is to consider the equation
\[
\xi=\frac{c_{i+1}p_{i}+p_{i-1}}{c_{i+1}q_{i}+q_{i-1}}
\]
and use it to express $c_{i+1}$ as
\begin{align*}
c_{i+1} &= \frac{-p_{i-1}+q_{i-1}\xi}{p_{i}-q_{i}\xi}
 = \frac{-p_{i-1}-q_{i-1}\delta'}{p_{i}+q_{i}\delta'}
 = \frac{-p_{i-1}-q_{i-1}\delta'}{p_{i}+q_{i}\delta'}\frac{p_{i}+q_{i}\delta}{p_{i}+q_{i}\delta} \\
&= \frac{-p_{i-1}p_{i}-p_{i}q_{i-1}\delta'-p_{i-1}q_{i}\delta-q_{i}q_{i-1}N(\delta)}{N(\delta_{i})} \\
&= \frac{-p_{i-1}p_{i}-p_{i}q_{i-1}\Tr(\delta)+\left( -p_{i-1}q_{i}+p_{i}q_{i-1} \right) \delta -q_{i}q_{i-1}N(\delta)}{N(\delta_{i})} \\
&= \frac{-T_{i}^{(a)}+(-1)^{i-1} \delta}{N(\delta_{i})}.
\end{align*}

To obtain the last equality here, we used the relation
\begin{equation}
\label{eq:pi1q1-rel}
p_{i}q_{i-1}-p_{i-1}q_{i}=(-1)^{i-1}.
\end{equation}

From this we can conclude that the expression for $T_{i}^{(a)}$ holds.

We proceed similarly for $T_{i}^{(b)}$, writing $c_{i+1}$ as
and use it to express $c_{i+1}$ as
\begin{align*}
c_{i+1} &= \frac{-p_{i-1}+q_{i-1}\xi}{p_{i}-q_{i}\xi}
 = \frac{-p_{i-1}-q_{i-1}\delta'}{p_{i}+q_{i}\delta'}
 = \frac{-p_{i-1}-q_{i-1}\delta'}{p_{i}+q_{i}\delta'}\frac{p_{i}+q_{i}\delta}{p_{i}+q_{i}\delta} \\
&= \frac{-p_{i-1}p_{i}-p_{i}q_{i-1}\delta'-p_{i-1}q_{i}\delta-q_{i}q_{i-1}N(\delta)}{N(\delta_{i})} \\
&= \frac{-p_{i-1}p_{i}-p_{i-1}q_{i}\Tr(\delta)+\left( -p_{i}q_{i-1}+p_{i-1}q_{i} \right) \delta' -q_{i}q_{i-1}N(\delta)}{N(\delta_{i})} \\
&= \frac{-T_{i}^{(b)}+(-1)^{i} \delta'}{N(\delta_{i})}.
\end{align*}
\end{proof}

%
%
%
%
%
%
%

\section{Approximations of the $N_{i}$'s}
\label{sec:ni}

In this section, we will be concerned with determining the values of $N_{i}$. We
start with a recurrence relation for the norms $N_{i}$. This result is a generalization
of Proposition~5c) in \cite{Ka}.

\begin{lemma}
\label{lem:Ni-rel}
For any squarefree integer $D>1$ and for each $i\in\N_{0}$, we have
\[\def\arraystretch{2.2}
N_{i}=\frac{\sqrt{\Delta}}{c_{i+1}}-\frac{N_{i-1}}{c_{i+1}^{2}}.
\]
\end{lemma}

\begin{proof}
The main idea of the proof is based on the definitions of $T_{i}^{(a)}$ and $T_{i}^{(b)}$ and two relations given by Lemma~\ref{lem:tiab}. Using the mentioned definition we can express $T_{i+1}^{(a)}$ as
\[
T_{i+1}^{(a)}=p_{i+1}p_{i} + p_{i+1}q_{i}\Tr(\delta)+q_{i}q_{i+1}N(\delta).
\]

We know that $p_{i+1}=u_{i+1}p_{i}+p_{i-1}$ and $q_{i+1}=u_{i+1}q_{i}+q_{i-1}$. If we use these linear recurrences in the previous equation, we get
\begin{align*}
T_{i+1}^{(a)} &= u_{i+1} \left( p_{i}^{2}+p_{i}q_{i}\Tr(\delta)+q_{i}^{2}N(\delta) \right) + p_{i-1} \left( p_{i}+q_{i} \Tr(\delta) \right) + q_{i}q_{i-1}N(\delta).
\end{align*}

Since $N \left( \delta_{i} \right) = p_{i}^{2}+p_{i}q_{i}\Tr(\delta)+q_{i}^{2}N(\delta)$ and $T_{i}^{(b)}=p_{i-1}p_{i}+p_{i-1}q_{i}\Tr(\delta)+q_{i}q_{i-1}N(\delta)$, we have
\[
T_{i+1}^{(a)}=u_{i+1}N \left( \delta_{i} \right) + T_{i}^{(b)}.
\]

Considering this relation for $i-1$, replacing $T_{i}^{(a)}$ and $T_{i-1}^{(b)}$ by the expressions given by Lemma~\ref{lem:tiab} and using \eqref{eq:ci}, we conclude that
\[
(-1)^{i+1}\delta - N \left( \delta_{i} \right)c_{i+1} = (-1)^{i}\xi - \left( u_{i}-c_{i} \right) N \left( \delta_{i-1} \right) = (-1)^{i}\xi - \frac{N \left( \delta_{i-1} \right)}{c_{i+1}}.
\]

Hence
\[
N \left( \delta_{i} \right)c_{i+1}
= (-1)^{i+1} \left( \delta+\xi \right) + \frac{N\left( \delta_{i-1} \right)}{c_{i+1}}.
\]

Since $N_{i}=(-1)^{i+1}N\left( \delta_{i} \right)$, we have
\[
N_{i}=\frac{\delta+\xi}{c_{i+1}}-\frac{N_{i-1}}{c_{i+1}^{2}}.
\]

From $\delta+\xi=\sqrt{\Delta}$, the lemma follows immediately.
\end{proof}

We proceed with upper and lower bounds on $N_{i}$. The inequalities in
Lemma~\ref{lem:upper} are analogous to Proposition~5d) and Proposition~6 of
\cite{Ka}.

\begin{lemma}
\label{lem:upper}
For each $i \in \N_{0}$, we have
\[\def\arraystretch{2.2}
\frac{N_{i}}{c_{i+2}}
<\frac{\sqrt{\Delta}}{2}.
\]

Moreover, we have
\begin{equation}
\label{eq:bnd2}
\frac{\sqrt{\Delta}}{c_{i+1}}\left(1-\frac{1}{c_{i}c_{i+1}}\right)<N_{i}
<\frac{\sqrt{\Delta}}{c_{i+1}}
\end{equation}
as well as
\begin{equation}
\label{eq:LB3}
\frac{\sqrt{\Delta}}{c_{i+1}+1} < N_{i}.
\end{equation}
\end{lemma}

\begin{proof}
From Lemma~\ref{lem:Ni-rel}, we see that
\[\frac{N_{i}}{c_{i+2}}<\frac{\sqrt{\Delta}}{c_{i+1}c_{i+2}}
=\frac{\sqrt{\Delta}}{\left( u_{i+1}+1/c_{i+2} \right) c_{i+2}}
=\frac{\sqrt{\Delta}}{u_{i+1}c_{i+2}+1}<\frac{\sqrt{\Delta}}{2}.
\]

This proves the first inequality in this lemma.

The upper bound in \eqref{eq:bnd2} follows directly from Lemma~\ref{lem:Ni-rel},
as we have
\[
N_{i}=\frac{\sqrt{\Delta}}{c_{i+1}}-\frac{N_{i-1}}{c_{i+1}^{2}}
<\frac{\sqrt{\Delta}}{c_{i+1}}.
\]

To prove the lower bound in \eqref{eq:bnd2}, we apply the upper bound we just
obtained to Lemma~\ref{lem:Ni-rel}:
\[
N_{i}=\frac{\sqrt{\Delta}}{c_{i+1}}-\frac{N_{i-1}}{c_{i+1}^{2}}
>\frac{\sqrt{\Delta}}{c_{i+1}}-\frac{\sqrt{\Delta}}{c_{i}}\frac{1}{c_{i+1}^{2}}.
\]

For the proof of \eqref{eq:LB3}, we proceed as follows. We start by using the lower bound in \eqref{eq:bnd2}. If $c_{i}>2$, then $1-1/\left( c_{i}c_{i+1} \right)>c_{i+1}/\left( c_{i+1}+1 \right)$ (in fact, even $c_{i+1} \left( c_{i}-1 \right)>1$ suffices). Hence \eqref{eq:LB3} follows.

Suppose $i=0$, then we have
\[
\frac{N_{0}}{\sqrt{\Delta}}=\frac{1}{c_{1}}-\frac{1}{c_{1}^{2} \sqrt{\Delta}},
\]
since $N_{-1}=1$. Now we want to show that
\[
\frac{1}{c_{1}}-\frac{1}{c_{1}^{2}\sqrt{\Delta}}
= \frac{1}{c_{1}} \left( 1-\frac{1}{c_{1}\sqrt{\Delta}} \right)
>\frac{1}{c_{1}+1}=\frac{1}{c_{1}} \left( 1-\frac{1}{c_{1}} \right).
\]

This is the same as showing that
\[
1-\frac{1}{c_{1}\sqrt{\Delta}} > 1-\frac{1}{c_{1}}.
\]

This holds when $c_{1}>1/ \left( \sqrt{\Delta}\, -1 \right)$ and since $c_{1}>1$, we require $\sqrt{\Delta}>2$. But this is always true since $\sqrt{\Delta}=2\sqrt{D} \geq 2\sqrt{2}>2.8$ for $D \equiv 2,3 \bmod 4$ and $\sqrt{\Delta}=\sqrt{D} \geq \sqrt{5}>2.2$ for $D \equiv 1 \bmod 4$.

Suppose $i \geq 1$. Applying  Lemma~\ref{lem:Ni-rel} twice, we have
\[
\frac{N_{i}}{\sqrt{\Delta}}
=\frac{1}{c_{i+1}}-\frac{1}{c_{i}c_{i+1}^{2}}+\frac{N_{i-2}}{c_{i}^{2}c_{i+1}^{2} \sqrt{\Delta}}
=\frac{u_{i}}{c_{i}c_{i+1}}+\frac{N_{i-2}}{c_{i}^{2}c_{i+1}^{2}\sqrt{\Delta}}
\geq \frac{1}{c_{i}c_{i+1}}+\frac{1}{c_{i}^{2}c_{i+1}^{2} \sqrt{\Delta}},
\]
since $N_{i-2}, u_{i} \geq 1$ and using \eqref{eq:ci}. So it remains to show that
\[
\frac{1}{c_{i}c_{i+1}}>\frac{1}{c_{i+1}+1}.
\]

This is the same as showing that $c_{i+1}+1>c_{i}c_{i+1}$, i.e., $c_{i+1} \left( c_{i}-1 \right)<1$, which we know is true by our assumption, completing the proof.
\end{proof}

\section{Approximations of the norms of indecomposable integers}
\label{sec:mi}


In the following proposition, we will determine when $N\left( \delta_{i,r} \right)$
takes its largest value for a fixed index $i$.

\begin{prop}
\label{prop:r}
Let $i$ be odd and $r_{0}$ be such that $N \left( \delta_{i,r_{0}}\right)$ has
the maximal value among $N\left( \delta_{i,r} \right)$ where $0\leq r\leq u_{i+2}$.

If $u_{i+2}$ is even, then $r_{0}=u_{i+2}/2$.

If $u_{i+2}$ is odd, then $r_{0}$ is one of $\left( u_{i+2} \pm 1 \right)/2$.
\end{prop}

\begin{proof}
We first express $\delta_{i+1}'\delta_{i}$ in terms of $T_{i}^{(a)}$ and $T_{i}^{(b)}$. From the definitions and \eqref{eq:pi1q1-rel}, we have
\begin{align*}
\delta_{i+1}'\delta_{i} &= \left( p_{i+1}+q_{i+1}\delta' \right) \left( p_{i}+q_{i}\delta \right) \\
&= p_{i+1}p_{i}+p_{i+1}q_{i}\delta+p_{i}q_{i+1}\delta'+q_{i}q_{i+1}N(\delta) \\
&= \frac{T_{i+1}^{(a)}+T_{i+1}^{(b)}}{2} + \left( p_{i+1}q_{i}-p_{i}q_{i+1} \right)\frac{\delta}{2} - \left( p_{i+1}q_{i}-p_{i}q_{i+1} \right)\frac{\delta'}{2} \\
&= \frac{T_{i+1}^{(a)}+T_{i+1}^{(b)}}{2} + (-1)^{i}\frac{\delta-\delta'}{2}.
\end{align*}

We use this relationship to obtain an expression for the norm, $N \left( \delta_{i,r} \right)$. Thus
\begin{align*}
N\left( \delta_{i,r} \right)
&= N\left( \delta_{i}+r\delta_{i+1} \right) = N \left( \frac{1}{\delta_{i+1}'}\delta_{i+1}' \left(\delta_{i}+r\delta_{i+1} \right)\right)=\frac{N \left(\delta_{i+1}'\delta_{i}+rN \left( \delta_{i+1} \right) \right)}{N \left( \delta_{i+1} \right)} \\
&= \frac{N\left(\frac{T_{i+1}^{(a)}+T_{i+1}^{(b)}}{2} + (-1)^{i}\frac{\delta-\delta'}{2}+rN \left( \delta_{i+1} \right) \right)}{N \left( \delta_{i+1} \right)} = \frac{\left( \frac{\delta-\delta'}{2} \right)^{2}-\left(\frac{T_{i+1}^{(a)}+T_{i+1}^{(b)}}{2}+rN \left( \delta_{i+1} \right)\right)^{2}}{\left|N \left( \delta_{i+1} \right) \right|},
\end{align*}
the last equality holding because $\frac{T_{i+1}^{(a)}+T_{i+1}^{(b)}}{2}+rN \left( \delta_{i+1} \right) \in \Q$, while $\frac{\delta-\delta'}{2}$ is a rational multiple of $\sqrt{D}$. Also notice that we use the fact that $N \left( \delta_{i+1} \right)<0$.

Hence
\[
N\left( \delta_{i,r} \right) = \frac{\left( \delta-\delta' \right)^{2} - \left( T_{i+1}^{(a)}+T_{i+1}^{(b)}+2rN \left( \delta_{i+1} \right) \right)^{2}}{4N_{i+1}}.
\]

It is clear from this expression that, for the fixed index $i$, $N\left( \delta_{i,r} \right)$ is maximal if the value of
$\left| T_{i+1}^{(a)}+T_{i+1}^{(b)}-2rN_{i+1} \right|$ is minimal. This is equivalent to the minimization of the value of 
\[
\left|\frac{T_{i+1}^{(a)}+T_{i+1}^{(b)}}{2N_{i+1}}-r\right|.
\]

Our next goal is to prove the estimate
\[
\left|\frac{T_{i+1}^{(a)}+T_{i+1}^{(b)}}{2N_{i+1}}-\frac{u_{i+2}}{2}\right|<\frac{1}{2}.
\]

From our expressions for $T_{i+1}^{(a)}$ and $T_{i+1}^{(b)}$ in Lemma~\ref{lem:tiab}, we obtain
\[
T_{i+1}^{(a)}+T_{i+1}^{(b)}
= (-1)^{i+2}(\delta+\xi)-2c_{i+2}N\left( \delta_{i+1} \right) = -\sqrt{\Delta}+2c_{i+2}N_{i+1},
\]
since $i$ is odd. Using this relationship and multiplying our desired inequality
by $2N_{i+1}$, we find that we want to show
\begin{align*}
N_{i+1} &> \left| T_{i+1}^{(a)}+T_{i+1}^{(b)}-N_{i+1}u_{i+2} \right|
= \left| -\sqrt{\Delta}+2c_{i+2}N_{i+1}-N_{i+1}u_{i+2} \right| \\
&= \left| -\sqrt{\Delta}+N_{i+1}(2c_{i+2}-u_{i+2}) \right|
= \left| -\sqrt{\Delta}+N_{i+1}\left(c_{i+2}+\frac{1}{c_{i+3}}\right) \right|.
\end{align*}

We show this inequality holds by applying the same procedure as in \cite{Ka} to
prove Proposition~10 there. We prove two inequalities.

a) $N_{i+1} > -\sqrt{\Delta} + N_{i+1}\left( c_{i+2}+\frac{1}{c_{i+3}} \right)$:

By the upper bound in the second inequality of Lemma~\ref{lem:upper}, \eqref{eq:ci}
and since $u_{i+2}>c_{i+2}-1$, we have 
\[
\sqrt{\Delta} > N_{i+1}c_{i+2} > N_{i+1}\left( c_{i+2}+\frac{1}{c_{i+3}}-1 \right),
\]
as we wanted to prove.

b) $N_{i+1} > \sqrt{\Delta}-N_{i+1}\left(c_{i+2}+\frac{1}{c_{i+3}}\right)$:

First note that from \eqref{eq:ci},
\[
c_{i+1}c_{i+2} = \left( u_{i+1}+\frac{1}{c_{i+2}}\right)c_{i+2}
= u_{i+1}c_{i+2}+1 \geq c_{i+2}+1,
\]
and so, upon dividing the left-most and right-most expressions by $c_{i+1}c_{i+2}^{2}$, we have
\[
-\frac{1}{c_{i+1}c_{i+2}} + \frac{1}{c_{i+2}} - \frac{1}{c_{i+1}c_{i+2}^{2}} \geq 0.
\]

Hence
\begin{align*}
\sqrt{\Delta} & \leq
\sqrt{\Delta} - \frac{\sqrt{\Delta}}{c_{i+1}c_{i+2}} + \frac{\sqrt{\Delta}}{c_{i+2}} - \frac{\sqrt{\Delta}}{c_{i+1}c_{i+2}^{2}}
=
\frac{\sqrt{\Delta}}{c_{i+2}}\left( 1-\frac{1}{c_{i+1}c_{i+2}} \right) \left( c_{i+2}+1 \right) \\
&< N_{i+1}\left( c_{i+2}+1 \right) < N_{i+1}\left( c_{i+2}+1+\frac{1}{c_{i+3}} \right),
\end{align*}
as we wanted to show (note that we used the lower bound for $N_{i+1}$ in
Lemma~\ref{lem:upper} to prove the penultimate inequality).

Let $r_{0}$ be an integer such that the value of
\[
\left| \frac{T_{i+1}^{(a)}+T_{i+1}^{(b)}}{2N_{i+1}}-r_{0} \right|
\]
is minimal. This quantity is at most $1/2$. Hence
\[
\left| r_{0}-\frac{u_{i+2}}{2} \right|
< \left| r_{0}-\frac{T_{i+1}^{(a)}+T_{i+1}^{(b)}}{2N_{i+1}} \right| + \left| \frac{T_{i+1}^{(a)}+T_{i+1}^{(b)}}{2N_{i+1}}-\frac{u_{i+1}}{2} \right|
< \frac{1}{2}+\frac{1}{2}=1,
\] 
which is precisely the assertion of the proposition.

If $u_{i+2}/2 \in \Z$, then $r_{0}=u_{i+2}/2$ follows immediately.
\end{proof}

In the following Proposition,
we will provide an approximation to $M_{i}$.

\begin{prop}
\label{prop:powerSeriesM}
Suppose that $i$ is odd.
Then
\begin{equation}
\label{eq:mi-conj1}
c_{i+2}-2 < \frac{4M_{i}}{\sqrt{\Delta}}<c_{i+2}+1.
\end{equation}
\end{prop}

\begin{proof}
First of all, we will derive an expression for $N\left( \delta_{i,r} \right)$.
We have
\begin{align*}
N\left( \delta_{i,r} \right)
& = \left( \delta_{i}+r\delta_{i+1} \right) \left( \delta'_{i}+r\delta'_{i+1} \right)
  = N \left( \delta_{i} \right) + r^{2}N \left( \delta_{i+1} \right)+r\left( \delta_{i}\delta_{i+1}'+\delta_{i+1}\delta' \right) \\
& = N \left( \delta_{i} \right) + r^{2}N \left( \delta_{i+1} \right)+r\left( 2p_{i}p_{i+1}+2q_{i}q_{i+1}N(\delta) + \left( p_{i+1}q_{i}+p_{i}q_{i+1} \right)\Tr(\delta) \right) \\
& = N \left( \delta_{i} \right) + r^{2}N \left( \delta_{i+1} \right)+r\left( T_{i+1}^{(a)}+T_{i+1}^{(b)} \right),
\end{align*}
the last equality following from the definitions of $T_{i+1}^{(a)}$ and $T_{i+1}^{(b)}$.

Since $i$ is odd, we have
\begin{equation}
\label{eq:Ndir}
N\left( \delta_{i,r} \right) = N_{i}-r^{2}N_{i+1}+r\left( T_{i+1}^{(a)}+T_{i+1}^{(b)} \right).
\end{equation}

For $r=u_{i+2}$, we get
\[
N_{i+2}=N_{i}-u_{i+2}^{2}N_{i+1}+u_{i+2}\left(T_{i+1}^{(a)}+T_{i+1}^{(b)}\right).
\]

Hence
\begin{equation}
\label{eq:uiti1-rel}
T_{i+1}^{(a)}+T_{i+1}^{(b)} = \frac{N_{i+2}+u_{i+2}^{2}N_{i+1}-N_{i}}{u_{i+2}}.
\end{equation}


From \eqref{eq:Ndir} with $r=u_{i+2}/2$, we obtain 
\[
M_{i}=N_{i}-\frac{u_{i+2}^{2}}{4}N_{i+1} + \frac{u_{i+2}}{2}\left( T_{i+1}^{(a)}+T_{i+1}^{(b)} \right).
\] 

Substituting \eqref{eq:uiti1-rel} into this equation, we conclude that
\begin{equation}
\label{eq:mi-a}
4M_{i}=4N_{i}-u_{i+2}^{2}N_{i+1} + 2\left( N_{i+2}+u_{i+2}^{2}N_{i+1}-N_{i} \right)
=u_{i+2}^{2}N_{i+1}+2\left( N_{i+2}+N_{i} \right).
\end{equation}

%

From \eqref{eq:Ndir} with $r=\left( u_{i+2}-1 \right)/2$
and \eqref{eq:uiti1-rel}, we obtain 
\begin{equation}
\label{eq:mi-bMinus}
4M_{i}=2\frac{u_{i+2}+1}{u_{i+2}}N_{i}+\left( u_{i+2}^{2}-1 \right)N_{i+1}
+2\frac{u_{i+2}-1}{u_{i+2}}N_{i+2}.
\end{equation}

%
%
%

For $r=\left( u_{i+2}+1 \right)/2$, we have
\begin{equation}
\label{eq:mi-bPlus}
4M_{i}=2\frac{u_{i+2}-1}{u_{i+2}}N_{i}+\left( u_{i+2}^{2}-1 \right)N_{i+1}
+2\frac{u_{i+2}+1}{u_{i+2}}N_{i+2}.
\end{equation}

We now use \eqref{eq:mi-a}, \eqref{eq:mi-bMinus}, \eqref{eq:mi-bPlus} above along with \eqref{eq:bnd2} and \eqref{eq:LB3} of Lemma~\ref{lem:upper}.

From \eqref{eq:mi-a} and Lemma~\ref{lem:Ni-rel}, along with \eqref{eq:LB3}, we have
\begin{align*}
4M_{i}
&=u_{i+2}^{2}N_{i+1}+2\left( N_{i+2}+N_{i} \right) \\
&= \left( u_{i+2}^{2}-2c_{i+2}^{2}-\frac{2}{c_{i+3}^{2}} \right) N_{i+1} +2\sqrt{\Delta}c_{i+2} + 2 \frac{\sqrt{\Delta}}{c_{i+3}} \\
&< \left( u_{i+2}^{2}-2c_{i+2}^{2}-\frac{2}{c_{i+3}^{2}} \right) \frac{\sqrt{\Delta}}{c_{i+2}+1} +2\sqrt{\Delta}c_{i+2} + 2 \frac{\sqrt{\Delta}}{c_{i+3}},
\end{align*}
if $u_{i+2}$ is even.

Substituting $u_{i+2}=c_{i+2}-1/c_{i+3}$, we have
\[
\frac{4M_{i}}{\sqrt{\Delta}}
< \frac{c_{i+2}^{2} + 2c_{i+2} + \left( -1/c_{i+3}^{2}+2/c_{i+3} \right)}{c_{i+2}+1}
= c_{i+2}+1 + \frac{-\left( 1/c_{i+3}-1 \right)^{2}}{c_{i+2}+1} < c_{i+2}+1.
\]

Using the upper bound for $N_{i}$ in \eqref{eq:bnd2} instead of the lower bound
in \eqref{eq:LB3}, we obtain
\[
\frac{u_{i+2}\left( 2c_{i+2}-u_{i+2} \right)}{c_{i+2}}<\frac{4M_{i}}{\sqrt{\Delta}}.
\]

Since $u_{i+2}\left( 2c_{i+2}-u_{i+2} \right)=\left( c_{i+2}-1 \right) c_{i+2}
+ \left( c_{i+2} - \left( c_{i+2}-u_{i+2} \right)^{2} \right) > \left( c_{i+2}-1 \right) c_{i+2}$,
our lower bound follows.

For $r=\left( u_{i+2}+1 \right)/2$, from \eqref{eq:mi-bPlus} and Lemma~\ref{lem:Ni-rel},
along with \eqref{eq:LB3}, we have
\begin{align*}
4M_{i}
&= \left( u_{i+2}^{2}-1 \right)N_{i+1}+2\frac{u_{i+2}-1}{u_{i+2}}N_{i}+2\frac{u_{i+2}+1}{u_{i+2}}N_{i+2} \\
&= \left( u_{i+2}^{2}-1-2\frac{\left( u_{i+2}+1 \right) \left( c_{i+2}-u_{i+2} \right)^{2}}{u_{i+2}}-2\frac{\left( u_{i+2}-1 \right) c_{i+2}^{2}}{u_{i+2}} \right) N_{i+1} \\
& +2\sqrt{\Delta}\frac{\left( u_{i+2}+1 \right) \left( c_{i+2}-u_{i+2} \right)}{u_{i+2}}+2\sqrt{\Delta}\frac{\left( u_{i+2}-1 \right) c_{i+2}}{u_{i+2}} \\
&< -\left( u_{i+2}-2c_{i+2}+1 \right)^{2} \frac{\sqrt{\Delta}}{c_{i+2}+1} \\
& +2\sqrt{\Delta}\frac{\left( u_{i+2}+1 \right) \left( c_{i+2}-u_{i+2} \right)}{u_{i+2}}+2\sqrt{\Delta}\frac{\left( u_{i+2}-1 \right) c_{i+2}}{u_{i+2}} \\
&= \frac{u_{i+2}^{2}-\left( 2c_{i+2}-4 \right) u_{i+2}-6c_{i+2}+3}{c_{i+2}+1}
=c_{i+2}+1-\frac{\left( u_{i+2}-c_{i+2}+2 \right)^{2}}{c_{i+2}+1}<c_{i+2}+1.
\end{align*}

As for $r=u_{i+2}/2$, we can use the upper bound for $N_{i}$ in \eqref{eq:bnd2}
instead of the lower bound in \eqref{eq:LB3} to obtain
\[
\frac{\left( u_{i+2}+1 \right) \left( 2c_{i+2}-u_{i+2}-1 \right)}{c_{i+2}}
<\frac{4M_{i}}{\sqrt{\Delta}}.
\]

Since
\[
\left( u_{i+2}+1 \right) \left( 2c_{i+2}-u_{i+2}-1 \right)
=\left( c_{i+2}-1 \right) c_{i+2}
+ \left( c_{i+2} - \left( c_{i+2}-u_{i+2}-1 \right)^{2} \right) > \left( c_{i+2}-1 \right) c_{i+2},
\]
our lower bound follows.

For $r=\left( u_{i+2}-1 \right)/2$, we proceed in the same way, using \eqref{eq:mi-bMinus},
Lemma~\ref{lem:Ni-rel} and \eqref{eq:LB3} to obtain
\[
4M_{i}
<c_{i+2}+1-\frac{\left( u_{i+2}-c_{i+2} \right)^{2}}{c_{i+2}+1}<c_{i+2}+1.
\]

Once again, we use the upper bound for $N_{i}$ in \eqref{eq:bnd2}
instead of the lower bound in \eqref{eq:LB3} to obtain
\[
\frac{\left( u_{i+2}-1 \right) \left( 2c_{i+2}-u_{i+2}+1 \right)}{c_{i+2}}
<\frac{4M_{i}}{\sqrt{\Delta}}.
\]

Since
\[
\left( u_{i+2}-1 \right) \left( 2c_{i+2}-u_{i+2}+1 \right)
=\left( c_{i+2}-2 \right) c_{i+2}
+ \left( 2u_{i+2}-1-\left( c_{i+2}-u_{i+2} \right)^{2} \right) > \left( c_{i+2}-2 \right) c_{i+2},
\]
our lower bound follows here too.

Note that it is only in this case where we need $c_{i+2}-2$ as our lower bound.
In fact, $c_{i+2}-1$ suffices, except for $1<c_{i+2}<2$ and $2.618\ldots<c_{i+2}<3$.
\end{proof}



\section{Computational Work}
\label{sec:calcs}

Initially, the techniques developed by Kala \cite{Ka} for $D \equiv 2 \bmod 4$ were adapted for use with $D \equiv 1, 3 \bmod 4$ too. In this way, we found
counterexamples to the conjecture of Jang and Kim for $D \equiv 1,3 \bmod 4$,
as well as a smaller counterexample for $D \equiv 2 \bmod 4$ than provided in
\cite{Ka}. For example, in this way, it was shown that
$D=68\,756\,796\,852\,765 \equiv 1 \bmod 4$ yields a counterexample to the
conjecture of Jang and Kim.

During this stage of the work, we directly examined small values of $D$. In this way, we found much smaller counterexamples for $D \equiv 1,2,3 \bmod 4$. Moreover we established that such examples were minimal.
In this section, we describe how such computations were done, the scope of the computations, and provide some summary information about the results.

For each $D$, we search for distinct odd indices $i$ and $j$ such that 
\begin{itemize}
\item $\delta_{i+1}$ is the element with the largest negative norm (= the smallest norm in absolute value $N_{i+1}=N$),
\item $N_{i+1} = \left| N \left( \delta_{i+1} \right) \right| < N_{j+1} = \left| N \left( \delta_{j+1} \right) \right|$, but the difference of the norms is small,
\item $M_{i}=N \left( \delta_{i, r} \right) < M_{j} = N \left( \delta_{j, t} \right)$ for $r$ and $t$ as in Proposition~\ref{prop:r}.
\end{itemize}

Such $M_{j}$ are counterexamples to the Jang-Kim Conjecture.

We performed two separate calculations. The counterexamples found from both calculations played a crucial role in the results of this paper.

\subsection{$1<D<10^{6}$}

First, for $1<D<10^{6}$, we found all $D$ giving rise to counterexamples. To find the period for all such $D$ required using very high precision. We set \verb+\p 2500+ in PARI/GP (the longest period found was for $D=950959$, which had period length $2448$) and the PARI stack size to be 64mb. This calculation took 19 hours using a development build of PARI/GP 2.12.0 \cite{Pari} on a
Windows 10 laptop with an Intel i7-3630QM processor and 8gb of RAM.

We found 54 counterexamples in total, 1 with $D \equiv 1 \bmod 4$, 29 with $D \equiv 2 \bmod 4$ and $24$ with $D \equiv 3 \bmod 4$. The minimal values of $D$ in each of these congruence classes that gives rise to a counterexample are given in Section~\ref{sec:min} below.

It is notable that there are significantly fewer counterexamples with $D \equiv 1 \bmod 4$ (the next such counterexample occurs at $D=1,332,413$). This behaviour continues as we search over larger ranges of $D$, as we will see in the description below of the second calculation.
We do not understand the reason for this behaviour. It also arises with the families of counterexamples that we have found and how fast they diverge from the Jang-Kim Conjecture.

\subsection{$1<D<10^{10}$}

For the second calculation, with $1<D<10^{10}$, we found all $D$ giving rise to counterexamples with the minimal period length of the continued fraction expansion of $\xi$ is at most $100$ (although many counterexamples were found where the period length was larger too). Much less accuracy was required here, \verb+\p 150+ was ample. This lead to the calculation being much faster and hence our ability to cover a much larger range. As in the first calculation, we set the PARI stack size to be 64mb.
On the same hardware and with the same version of PARI/GP, this calculation took approximately 580 CPU hours. This was spread across three of the four cores of the Intel i7-3630QM processor used.

We summarise in Table~\ref{table:count} information about the number of distinct squarefree $D<1-^{10}$ in each congruence class modulo $4$ with the minimal period length of the continued fraction expansion of $\xi$ at most $100$ that give rise to counterexamples. We provide two columns for $D \equiv 1 \bmod 4$. The first one is for the Jang-Kim Conjecture as stated, while the second column (marked with an asterisk in the header column) is for the conjecture as we believe they intended it. The minimal values of $D$ in each of these congruence classes that gives rise to a counterexample are given in Section~\ref{sec:min} below.

\begin{table}[h!]
\centering
\begin{tabular}{|c|r|r|r|r|} 
 \hline
 $N$ & $D \equiv 1 \bmod 4$ & $D \equiv 1 \bmod 4$ (*) & $D \equiv 2 \bmod 4$ & $D \equiv 3 \bmod 4$ \\ \hline
 $10^{6}$  &    $1$ &   $28$  &   $29$ &    $24$ \\
 $10^{7}$  &    $8$ &  $138$  &  $154$ &   $176$ \\
 $10^{8}$  &   $49$ &  $636$  &  $682$ &   $793$ \\
 $10^{9}$  &  $214$ & $2601$  & $2679$ &  $2966$ \\
 $10^{10}$ &  $879$ & $9648$  & $9210$ & $10626$ \\ \hline
\end{tabular}
\label{table:count}
\caption{Counterexample Counts with $D \leq N$}
\end{table}

A quantity that is of particular interest in us in this work is
\[
r(D)=\frac{M-UB_{JK}}{\sqrt{D}},
\]
where we define $UB_{JK}$ to be what we believe to be the intended upper bound in the Jang-Kim Conjecture, namely,
\[
\def\arraystretch{2.2}
UB_{JK}=
\left\{
\begin{array}{ll}
	\displaystyle\frac{D-a^{2}}{4N} & \text{if $D \equiv 1 \bmod 4$}, \\
	\displaystyle\frac{D-a^{2}}{N} & \text{if $D \equiv 2,3 \bmod 4$},
\end{array}
\right.
\]
with $N$ as in Theorem~\ref{thm:jk} and $a$ as in Conjecture~\ref{conj:MP}. We record here the
largest values of $r(D)$ that we found in each congruence class.

For $D \equiv 1 \bmod 4$, the largest value of $r(D)$ we found was $0.122981$ for $D=259,209,905$. The minimal period length of the continued fraction expansion of $\xi$ is 18 (note that using the upper bound actually stated by Jang and Kim in their conjecture rather than $UB_{JK}$, the largest value of $r(D)$ we found was $0.021724$ for $D=30,386,757$. The minimal period length of the continued fraction expansion of $\xi$ is 18).

For $D \equiv 2 \bmod 4$, the largest value of $r(D)$ we found was $0.242079$ for $D=34,650,842$. The minimal period length of the continued fraction expansion of $\xi$ is 20.

Although both these values of $D$ are considerably smaller than $10^{10}$, this is not significant. The second largest value of $r(D)$ for $D \equiv 2 \bmod 4$ is $0.240347$, which occurs for $D=9,720,174,694$. Similarly for $D \equiv 1 \bmod 4$, the next seven largest values of $r(D)$ arise from $D>10^{9}$.

For $D \equiv 3 \bmod 4$, the largest value of $r(D)$ we found was $0.243264$ for $D=3,555,318,415$. The minimal period length of the continued fraction expansion of $\xi$ is 14.

\subsection{$j$}

In the examples in the following sections, $N_{j}$ is the second largest value of the negative norms among the $\delta_{i}$'s. However, this need not always be the case, and it seems likely that there can be arbitrarily many such negative norms between the largest one and the one associated with the indecomposable number of largest norm.

The most extreme example we found for $D<10^{9}$ was $D=457,859,058$. Here we have $j=9$, and $N \left( \delta_{j} \right)$ is the fifth largest value of $N \left( \delta_{i} \right)<0$.

For $D<10^{10}$, the most extreme example arises from $D=5,654,211,695$, where $j=49$, and $N \left( \delta_{j} \right)$ is the eighth largest value of $N \left( \delta_{i} \right)<0$.

The occurrence of such counterexamples so far from the $\delta_{i}$ of largest negative norm, along with other information about their counts, etc. acquired from our calculations in Section~\ref{sec:calcs}, supports our claim that they can occur arbitrarily far from such elements.

\section{Minimal Counterexamples}
\label{sec:min}

\subsection{$D \equiv 1 \bmod 4$}

The smallest $D \equiv 1 \bmod 4$ for which we found a counterexample was $D=715461$. Here the continued fraction expansion of $\xi$ is
\[
\left[ 422, \overline{2, 2, 1, 4, 2, 9, 2, 281, 2, 9, 2, 4, 1, 2, 2, 845} \right],
\]
which has minimal period length $16$.  We obtain $N$ from $i=13$, for which $\delta_{i+1}=106901916571+126384120\sqrt{D}$, $N=N_{i+1}=359$. $\delta_{i,r}=\delta_{i,1}=$\\ $289692067643/2+342486629/2\sqrt{D}$, $M_{i}=487$. For such $D$ and $N$, we have $a=127$, so the conjectured Jang-Kim upper bound is $487$.

However, for $j=3$, we have $N_{j+1}=365$, $\delta_{j,t}=\delta_{j,1}=16917+20\sqrt{D}$, $M_{j}=489$, which exceeds the Jang-Kim upper bound of $487$.

\subsection{$D \equiv 1 \bmod 4$ (*)}

As in the previous section, we use the asterisk here to indicate that this is based on how we believe the Jang-Kim Conjecture should have been formulated. Namely, with $a$ the smallest non-negative integer such that $a^{2} \equiv D \bmod 4N$ (rather than $\bmod N$, as they wrote).

The smallest $D \equiv 1 \bmod 4$ for which we found a counterexample was $D=12441$. Here the continued fraction expansion of $\xi$ is
\[
\left[ 55, \overline{3, 1, 2, 2, 3, 2, 2, 1, 3, 111} \right],
\]
which has minimal period length $10$.  We obtain $N$ from $i=3$, for which $\delta_{i+1}=1450+13\sqrt{D}$, $N=N_{i+1}=29$. $\delta_{i,r}= \delta_{i,1}=4127/2+37/2\sqrt{D}$, $M_{i}=100$. For such $D$ and $N$, we have $a=29$, so the conjectured Jang-Kim upper bound is $100$.

However, for $j=7$, we have $N_{j+1}=30$, $\delta_{j,t}=\delta_{j,1}=66812+599\sqrt{D}$, $M_{j}=103$, which exceeds the Jang-Kim upper bound of $100$.

\subsection{$D \equiv 2 \bmod 4$}

The smallest $D \equiv 2 \bmod 4$ for which we found a counterexample was $D=25982$. Here the continued fraction expansion of $\xi$ is
\[
\left[ 161, \overline{5, 3, 1, 1, 4, 1, 1, 1, 2, 1, 1, 4, 1, 7, 1, 1, 1, 28, 1, 1, 1, 7, 1, 4, 1, 1, 2, 1, 1, 1, 4, 1, 1, 3, 5, 322} \right],
\]
which has minimal period length $36$. We obtain $N$ from $i=33$, for which $\delta_{i+1}=241149629719159+1496064474059\sqrt{D}$, $N=N_{i+1}=61$. $\delta_{i,r}=\delta_{i,2}=550253365757040+3413708381801\sqrt{D}$, $M_{i}=418$. For such $D$ and $N$, we have $a=22$, so the conjectured Jang-Kim upper bound is $418$.

However, for $j=3$, we have $N_{j+1}=62$, $\delta_{j,t}=\delta_{j,2}=15313+95\sqrt{D}$, $M_{j}=419$, which exceeds the Jang-Kim upper bound of $418$.

\subsection{$D \equiv 3 \bmod 4$}

The smallest $D \equiv 3 \bmod 4$ for which we found a counterexample was $D=46559$. Here the continued fraction expansion of $\xi$ is
\[
\left[ 215, \overline{1, 3, 2, 4, 1, 1, 1, 2, 1, 1, 2, 215, 2, 1, 1, 2, 1, 1, 1, 4, 2, 3, 1, 430} \right],
\]
which has minimal period length $24$. We obtain $N$ from $i=9$, for which $\delta_{i+1}=187293+868\sqrt{D}$, $N=N_{i+1}=167$. $\delta_{i,r}=\delta_{i,1}=295828+1371\sqrt{D}$, $M_{i}=265$. For such $D$ and $N$, we have $a=48$, so the conjectured Jang-Kim upper bound is $265$.

However, for $j=1$, we have $N_{j+1}=175$, $\delta_{j,t}=\delta_{j,1}=1079+5\sqrt{D}$, $M_{j}=266$, which exceeds the Jang-Kim upper bound of $265$.

\section{Infinite Families}
\label{sec:family}

As well as finding the smallest $D$ for which the conjecture of Jang and Kim fails, we also found infinitely many examples for which this conjecture fails. Moreover, we also found families where the difference between the maximum norm of an indecomposable integer and the conjectured upper bound grows arbitrarily large.

\subsection{$D \equiv 3 \bmod 4$}

Let $m$ and $n$ be non-negative integers and set
\begin{eqnarray*}
D=D(m,n) 
       & = & \left( 2\left( 256n^{2}+672n+429 \right)^{2} m + 3584n^{3}+13376n^{2}+16414n+6641 \right) \\
       &   & \times \left( 2\left( 128n^{2}+328n+203 \right)^{2} m + 896n^{3}+3232n^{2}+3822n+1487 \right).
\end{eqnarray*}

Putting $u_{0}=2\left( 128n^{2}+328n+203 \right) \left( 256n^{2}+672n+429 \right) m+1792n^{3}+6576n^{2}+7922n+3142$,
note that the continued fraction expansion of $\sqrt{D(m,n)}$ is
\[
\left[ u_{0}, \overline{2, 8n+10, 1, 4n+3, 1, 1, 1, 1, 1, 4n+3, 1, 8n+10, 2, 2u_{0}} \right].
\]

We find that the minimum negative norm of the elements $\delta_{i+1}$ occurs for $i=5$ and its absolute value is
\[
N_{m,n}=
4 \left( 128n^{2}+328n+203 \right)^{2}m
+1792n^{3}+6464n^{2}+7644n+2974.
\]

Since the only common factor of the linear and constant coefficients of $N_{m,n}$ (viewed as a polynomial in $m$) is $2$, by Dirichlet's Theorem on primes in arithmetic progression, for each fixed value of $n$, there are infinitely
values of $m$ such that $N_{m,n}$ is two times a prime. For such values of $m$,
$N_{m,n}$ is twice the second factor in the expression for $D(m,n)$, while the first factor is odd.
Therefore, $D_{m,n} \equiv N_{m,n}/2 \bmod N_{m,n}$ and so $D_{m,n} \equiv \left( N_{m,n}/2 \right)^{2} \bmod N_{m,n}$, since $N_{m,n} \equiv 2 \bmod 4$. Thus we can take $a_{m,n}=N_{m,n}/2$.
Therefore for such $m$, the upper
bound of Jang and Kim is
\begin{eqnarray*}
& & 
UB_{JK}=\frac{D(m,n)-N_{m,n}^{2}/4}{N_{m,n}}
 =\frac{D(m,n)}{N_{m,n}}-\frac{N_{m,n}}{4} \\
&=&
\left( 49152n^{4}+260096n^{3}+511680n^{2}+443408n++142832 \right)m +1344n^{3}+5072n^{2}+6296n+2577.
\end{eqnarray*}

However for $j=11$, we find that the norm of $\delta_{j,1}$ is
\[
M_{m,n}=\left( 65536n^{4}+344064n^{3}+669952n^{2}+573216n+181892 \right)m +1792n^{3}+6688n^{2}+8172n+3282,
\]
which exceeds the upper bound of Jang and Kim. 
A comparison with the above expression for $u_{0}$ then shows that
$r(D) \rightarrow 1/4$, as $m, n \rightarrow \infty$. In fact, we have
\[
r(D)
< \frac{16384n^{4}+83968n^{3}+158272n^{2}+129808n+39060}{2\left( 128n^{2}+328n+203 \right) \left( 256n^{2}+672n+429 \right)}
< \frac{1}{4}-\frac{1}{64n},
\]
so $1/4$ is approached from below.

Lastly, we note that one can show that $M_{m,n}$ is the largest norm of any indecomposable element of $\Q \left( \sqrt{D(m,n)} \right)$ too.

We now explain how we discovered such a family.

Using the counterexamples we collected from the calculations described in Section~\ref{sec:calcs}, we observed that for all the counterexamples with $r(D)>0.2$ for $D \equiv 3 \bmod 4$, the minimal period length of the continued fraction expansion of $\xi$ was $14$. Moreover the largest value of $r(D)$ found for all $D \equiv 3 \bmod 4$ with $1<D<10^{10}$ and period length at most $100$ arose from a counterexample ($D=3,555,318,415$) with  was $r(D)=0.243264$ and period length $14$ too. Since short periods make our creation and checking of families easy, we looked at patterns among such counterexamples.

All the counterexamples we had with period length $14$ and $r(D)>0.2$ had the same pattern to the continued fraction expansion of $\xi$, namely,
\[
\left[ u_{0}, \overline{2, u_{2}, 1, u_{4}, 1, 1, 1, 1, 1, u_{4}, 1, u_{2}, 2, 2u_{0}} \right].
\]

%


So we searched over all $D$ such that the continued fraction of $\xi$ was
of this form with $0<u_{2}, u_{4}<100$. We noticed that the largest value of $r(D)$ occurred
when $u_{2}=2u_{4}+4$ and used such examples. Lastly, we noticed that $D$ was
smaller when $u_{4} \equiv 3 \bmod 4$, so we used only such $u_{4}$.

Using Maple, we found that if
\[
\xi = \left[ x, \overline{2, 2n+4, 1, n, 1, 1, 1, 1, 1, n, 1, 2n+4, 2, 2x} \right],
\]
then
\[
D=x^{2}+
4\frac{32n^{4}+272n^{3}+818n^{2}+1020n+449}{\left( 16n^{2}+72n+69
 \right)\left( 8n^{2}+34b+29 \right)} x
+\frac{\left( 8n^{2}+34n+31 \right)  \left( 4n^{2}+16n+13 \right)}{ \left( 16n^{2}+72n+69 \right)  \left( 8n^{2}+34n+29 \right)}.
\]

It can be shown that if
%
\[
x=\left( 16n^{2}+72n+69
 \right)\left( 8n^{2}+34b+29 \right)x'+\left( 28n^{3}+159n^{2}+541n/2+287/2 \right),
\]
then $D \in \Z$. Substituting $n=4n+3$ and $x'=2m$ (the use of $2m$ here ensures the two factors in the above expression for $D(m,n)$ are relatively prime) into the resulting
expression for $D$, we obtain the expression for $D(m,n)$ at the start
of this subsection.

\subsection{$D \equiv 2 \bmod 4$}

Let $m$ and $n$ be positive integers and set
\begin{eqnarray*}
D=D(m,n)
& = & \left( \left( 56n^{2}+184n+151 \right)^{2} m + 1568n^{4}-5376n^{3}-52960n^{2}-102576n-60832 \right) \\
&   & \times \left( \left( 112n^{2}+352n+276 \right)^{2} m + 6272n^{4}-23296n^{3}-202592n^{2}-365792n-203234 \right).
\end{eqnarray*}

Putting
\[
u_{0}=\left( 56n^{2}+184n+151 \right) \left( 112n^{2}+352n+276 \right) m+3136n^{4}-11200n^{3}-103640n^{2}-193808n-111190,
\]
note that the continued fraction expansion of $\sqrt{D(m,n)}$ is
\[
\left[ u_{0}, \overline{1, n, 1, 1, 1, 1, 1, 7n+10, 2, 7n+10, 1, 1, 1, 1, 1, n, 1, 2u_{0}} \right].
\]

We find that the minimum negative norm of the elements $\delta_{i+1}$ occurs for $i=3$ ($M_{3}$ comes from $\delta_{i,0}$), but for $j=7$, we have $u_{j+2}=2$, so we consider $\delta_{7,1}$ and find that
\[
M_{7,1}=\left( 56n^{2}+184n+151 \right)^{2}+3136n^{4}-10752n^{3}-105920n^{2}-205152n-121664.
\]

We have
\[
M_{7,1}-UB_{JK}
= \left( 784n^{4}+4592n^{3}+10027n^{2}+9672n+3477 \right)m
+784n^{4}-3248n^{3}-23545n^{2}-37678n-18553.
\]

A comparison with the above expression for $u_{0}$ shows that
$r(D) \rightarrow 1/4$, as $m, n \rightarrow \infty$. In fact, we have
\[
r(D) < \frac{1}{4}-\frac{1}{7n},
\]
so $1/4$ is approached from below.

\subsection{$D \equiv 1 \bmod 4$}

Here the limit of $r(D)$ within families is much smaller. This is in line
with the findings from our calculations for $D<10^{10}$ too. At the
moment, the best we have is the following single example.

For $D=172471024674149$, the continued fraction expansion of $\xi$ is
\[
\left[ 6566410, \overline{1, 2, 2, 3, 1, 1, 2, 20, 1, 20, 2, 1, 1, 3, 2, 2, 1, 13132821} \right],
\]
with a period length of $18$.

For $j=1$, $M_{j}=8569429$, $UB_{JK}=8279004$, so $r(D)=0.022114\ldots$.

Note that $N$ occurs for $i=5$.

%
%

\subsection{$D \equiv 1 \bmod 4$ (*)}

Let $m$ and $n$ be positive integers and set
\begin{eqnarray*}
D=D(m,n)
& = & \left( 2\left( 64n^{2}+144n+69 \right)^{2} m + 896n^{3}+2656n^{2}+2398n+685 \right) \\
&   & \times \left( 2\left( 32n^{2}+68n+29 \right)^{2} m + 224n^{3}+608n^{2}+486n+121 \right).
\end{eqnarray*}

Putting
\[
u_{0}=\left( 64n^{2}+144n+69 \right) \left( 32n^{2}+68n+29 \right)m+224n^{3}+636n^{2}+541n+143,
\]
note that the continued fraction expansion of $\sqrt{D(m,n)}$ is
\[
\left[ u_{0}, \overline{2, 4n+4, 1, 2n 1, 1, 1, 1, 1, 2n, 1, 4n+4, 2, 2u_{0}+1} \right].
\]

We find that the minimum negative norm of the elements $\delta_{i+1}$ occurs
for $i=5$ ($M_{5}$ comes from $\delta_{i,0}$), but for $j=11$, we have $u_{j+2}=2$,
so we consider $\delta_{11,1}$ and find that
\[
M_{11,1}=\left( 2048n^{4}+9216n^{3}+14624n^{2}+9576n+2206 \right) m +224n^{3}+664n^{2}+582n+159.
\]

We have
\[
M_{11,1}-UB_{JK}
= \left( 512n^{4}+2176n^{3}+3080n^{2}+1612n+246 \right) m
+56n^{3}+152n^{2}+104n+18.
\]

A comparison with the above expression for $u_{0}$ shows that
$r(D) \rightarrow 1/8$, as $m, n \rightarrow \infty$. In fact, we have
\[
r(D) < \frac{1}{8}-\frac{1}{64n},
\]
so $1/8$ is approached from below.

\section{Conjecture~\ref{conj:MP} and Theorem~\ref{thm:MP}}
\label{sec:conj}

\subsection{$D \equiv 1 \bmod 4$}

Our interest in the particular families of counterexamples in Section~\ref{sec:family} for $D \equiv 1 \bmod 4$ is that if we let $n \rightarrow \infty$, then
we find that $r(D) \rightarrow 1/8$ from below. We know of no counterexamples where $r(D)$ exceeds $1/8$.
None were found from our calculations described in Section~\ref{sec:calcs} (as noted above, the largest value of $r(D)$ that we found for $1<D<10^{10}$ was $r(D)=0.122981\ldots$ for $D=259,209,905$).
Furthermore, all the examples found with large $r(D)$ came from relatively short periods (at most of length 34), so we believe it is not likely that we missed larger values of $r(D)$ by the restriction of our calculation of those $D<10^{10}$.
Nor did any of our searches for infinite families of counterexamples lead to any such examples. Hence we make Conjecture~\ref{conj:MP} for $D \equiv 1 \bmod 4$.

\subsection{$D \equiv 2,3 \bmod 4$}

The justification for Conjecture~\ref{conj:MP} for $D \equiv 2,3 \bmod 4$ is of the same nature as that for $D \equiv 1 \bmod 4$. We chose the families of counterexamples in Section~\ref{sec:family} for $D \equiv 2,3 \bmod 4$ because $r(D) \rightarrow 1/4$ from below as $n \rightarrow \infty$. We know of no counterexamples where $r(D)$ exceeds $1/4$, neither from our calculations (as noted above, the largest value of $r(D)$ that we found for $1<D<10^{10}$ was $0.243264$ for $D= 3555318415$) or from our searches for infinite families of counterexamples.


We have not been able to prove this conjecture, but we have been able to prove Theorem~\ref{thm:MP}, which provides an upper bound that does grow like $\sqrt{D}$ and with a modest constant.

\subsection{Proof of Theorem~\ref{thm:MP}}

From the lower bound in \eqref{eq:LB3} and the upper bound in \eqref{eq:bnd2},
we see that $N_{i+1}<N_{j+1}$ implies that $c_{j+2}<c_{i+2}+1$.

Combining this with \eqref{eq:mi-conj1}, we obtain
\[
\frac{4M_{j}}{\sqrt{\Delta}}<c_{j+2}+1 \leq c_{i+2}+2.
\]

Also from \eqref{eq:mi-conj1}, we have
\[
c_{i+2}-2<\frac{4M_{i}}{\sqrt{\Delta}}.
\]

Hence
\[
\frac{4M_{j}-4M_{i}}{\sqrt{\Delta}}<4.
\]

If $D \equiv 2, 3 \bmod 4$, then we have $\sqrt{\Delta}=2\sqrt{D}$, so $M_{j}-M_{i}<2\sqrt{D}$.
Similarly, if $D \equiv 1 \bmod 4$, then $M_{j}-M_{i}<\sqrt{D}$. Combining the resulting upper bound for $M_{j}-M_{i}$ with Theorem~2(a) of \cite{Ka}, where $i$ is the index such that $N \left( \delta_{i+1} \right)$ is maximal among all  the negative norms, the theorem follows (note that Kala uses $\alpha_{i}$ and $\alpha_{i,r}$ where we, following Dress and Scharlau \cite{DS}, use $\delta_{i}$ and $\delta_{i,r}$, respectively).


\section*{Acknowledgements}

The authors gratefully acknowledge the many helpful suggestions of V\'it\v ezslav Kala during the preparation of this paper.

\end{document}